\documentclass{amsart}

\usepackage{
amsfonts,
latexsym,
amssymb
}

\newcommand{\labbel}{\label}

\newtheorem{theorem}{Theorem}[section]
\newtheorem{lemma}[theorem]{Lemma}

\newtheorem{proposition}[theorem]{Proposition} 
 
\newtheorem{corollary}[theorem]{Corollary}

\newtheorem*{theorem*}{Theorem}
\newtheorem*{corollary*}{Corollary}

\theoremstyle{definition}

\theoremstyle{remark}
\newtheorem{remark}[theorem]{Remark}

\newcommand{\brfrt}{\hspace{0 pt}}
\DeclareMathOperator{\cf}{cf}
\DeclareMathOperator{\CAP}{CAP}

\begin{document}

\title[Ultrafilter convergence in ordered  topological spaces]{Ultrafilter convergence in ordered \\ topological spaces}

\author{Paolo Lipparini} 
\address{Dipartimento di Matematica\\ Viale degli Ordini
 Scientifici\\II Universit\`a di Roma (Tor Vergata)\\I-00133 ROME ITALY}
\urladdr{http://www.mat.uniroma2.it/\textasciitilde lipparin}
\email{lipparin@axp.mat.uniroma2.it}

\keywords{Linearly ordered,  generalized ordered topological space; 
ultrafilter convergence, compactness, pseudocompactness; 
(pseudo-)gap; converging $\nu$-sequence;
(weak) $[\nu, \nu]$-compactness; 
(weak) initial $\lambda$-compactness; complete accumulation point; $\lambda$-boundedness; decomposable, descendingly complete, regular ultrafilter}

\subjclass[2010]{54F05, 54A20, 54B10, 54D20; 06A05, 54A35, 03E75}

\begin{abstract}
We characterize ultrafilter convergence and ultrafilter 
compactness in linearly ordered and generalized ordered topological spaces.
In such spaces, and for every ultrafilter $D$, 
the notions of $D$-compactness and of $D$-pseudocompactness
are equivalent.
Any product of initially $\lambda$-compact 
generalized ordered topological spaces is still  initially $\lambda$-compact.
On the other hand, preservation under products
of  certain  compactness properties is independent
from  the usual axioms for set theory.
\end{abstract} 
 
\maketitle

{\it 
\noindent
\hskip 70 pt Once upon a time people red papers and sometimes quoted them.

\noindent
\hskip 70 pt  Now people quote papers and sometimes read them.
}

\bigskip

\section{Introduction} \labbel{intro}

It is well-known that many covering properties,
which in general are distinct, turn out to be equivalent
for \emph{linearly ordered topological spaces} (henceforth, LOTS, for short).
For example, a LOTS is pseudocompact if and only if it is countably compact, if and only if
it is sequentially compact. See 
Gulden, Fleischman and Weston  \cite{GFW} and
Purisch \cite{purisch} for these and more general results,
also involving uncountable cardinals and dealing with product theorems, as well.
Most of the above results generalize 
to \emph{GO spaces}, short for \emph{generalized ordered topological spaces};  
 see \cite{purisch} and  \cite{goic}, or below.

Here we show that, for every ultrafilter $D$, $D$-compactness and 
$D$-\brfrt pseudocompactness are equivalent in GO spaces (Remark \ref{iff}, or Theorem \ref{thm}(1) $\Leftrightarrow $  (5)).
The particular case when $D$ is  an ultrafilter over $ \omega$
is due to  Garc{\'{\i}}a-Ferreira and Sanchis \cite{GFS},
whose paper contains many other related theorems.
More generally, we characterize 
$D$-converging sequences in GO spaces (Proposition \ref{single}(7)-(9)).
We also show that the $D$-compactness
of some GO space $X$ depends exclusively 
on
the ``decomposability spectrum''   of $D$
and
on the cardinal types of (pseudo-)gaps in $X$,
equivalently, on the existence of non converging
strictly monotone sequences of certain cardinal order types (Theorem \ref{thm}).

Some of the 
results mentioned in the first
paragraph can be obtained as corollaries of Theorem \ref{thm}.
For example, we get a proof that any product of initially $\lambda$-compact 
GO spaces is still  initially $\lambda$-compact (Corollary \ref{initial}).
The result is quite surprising, since noncompact covering properties
are usually not preserved by taking products. Just to state a significant example,
 the square of a Lindel\"of LOTS is not necessarily Lindel\"of.
The Sorgenfrey line is an example which is a GO space.
To get an example which is a LOTS, consider 
$\mathbb R \times \omega $ 
with the lexicographic order.
Also, it is well-known that the product of two
countably compact topological spaces is not necessarily countably compact,
and many similar counterexamples are known for initial
$\kappa$-compactness, when $\kappa$ is regular.
See, e.~g.,  Stephenson \cite{ste84}, Vaughan \cite{vau84},  Nyikos and Vaughan \cite{NV} and more references there.  
Hence, for $\kappa$ regular,  preservation of
initial
$\kappa$-compactness under products  is really a special property of GO spaces.
In passing, let us mention that, on the other hand,
when $\kappa$ is a singular strong limit cardinal,
every product of  initially
$\kappa$-compact topological spaces is 
still initially
$\kappa$-compact, by the celebrated  
 Stephenson and Vaughan Theorem 1 in \cite{sv74}. 

In the particular case of  LOTS, preservation of
initial
$\kappa$-compactness under products had originally been proved 
in \cite{GFW}, for every infinite cardinal $\kappa$. Not aware of 
Gulden, Fleischman and Weston's result, 
we originally have found a proof which holds for GO spaces, too,
and which uses ultrafilter convergence \cite{goic}. Later we realized that
 the theorem 
can also be proved by adapting some  arguments from
\cite{GFW} (see the second proof of Corollary \ref{initial} here), hence the use of ultrafilters becomes unnecessary.
However, luckily for scholars  fond of ultrafilters,
there are indeed compactness properties whose preservation
under products (in GO spaces) involves ultrafilters 
 in an essential way. 
For example, this is the case when considering simultaneously
countable compactness and  
$[ \lambda , \lambda ]$-\brfrt compactness.
Considering preservation
of such properties leads to statements independent from
the usual axioms of set theory, statements which involve the existence 
of ultrafilters with an ``unusual'' descending completeness spectrum
(equivalently, decomposability spectrum).
See Corollary \ref{gocorbb} for the general statement, and 
Corollary \ref{gocorbbb} for an explicit example.

Now a few words about the fate of \cite{goic}. 
Since all the results 
proved there are subsumed by  the present paper,
we are not going to submit \cite{goic} elsewhere.
However, we shall keep it available in archived form,
since it might be useful for those looking for 
a  direct proof of Corollaries \ref{initial} and  \ref{cor} here,
which essentially were  the main theorems of \cite{goic}.

\section{Preliminaries} \labbel{prel} 

We now recall the relevant definitions.
A \emph{LOTS} is a linearly ordered set endowed with the open interval topology.
A \emph{GO space} is a linearly ordered set 
with a $T_2$ topology
 having a base of order-convex sets. GO spaces are exactly 
spaces which can be obtained as subspaces of 
LOTS; the notion is more general, since if $X$ is a LOTS, and $Y \subseteq X$,
then  the subset topology on $Y$ induced by the topology on $X$
might be finer than the order topology  on $Y$ relative to the restriction
of the order on $X$. See, e.~g., \cite{BL} for more informations and references
about LOTS and GO spaces.

 If $D$ is an ultrafilter over some set $I$, then a topological space $X$ is said to be \emph{$D$-compact}
  if every $I$-indexed sequence $(x_i) _{i \in I} $ of elements of $X$ 
 \emph{$D$-converges} to some $x \in X$, that is,
 $\{ i \in I \mid x_i \in U\} \in D$,
for every open neighborhood $U$ of $x$.
The notion of ultrafilter convergence has proved
particularly useful in the study of compactness properties
of topological spaces, in particular with respect to preservation under products.
Classical papers on the subject are Bernstein \cite{ber70}, 
Ginsburg and Saks \cite{GS}, and
Saks \cite{sak78}. 
Excellent surveys of results proved until
the mid '80's are  Vaughan \cite{vau84} and Stephenson \cite{ste84}. 
Further results can be found in
Caicedo \cite{Ca} and \cite{sssr},
together with additional references.
 
Ginsburg and Saks \cite{GS} made also  a very effective 
use of  the notion of a $D$-limit point of 
a sequence of subsets of a topological space, and introduced 
the notion of $D$-pseudocompactness. 
See also \cite{gf99}.
If $(Y_i) _{i \in I} $  is a sequence of subsets of $X$,
a point $x \in X$ is said to be 
a \emph{$D$-limit point} of $(Y_i) _{i \in I} $
if $ \{ i \in I \mid  U \cap Y_i  \not= \emptyset    \} \in D $,
for every neighborhood $U$ of $x$
(thus when each $Y_i$ is a singleton we get back the notion
of $D$-convergence).
A topological space $X$ is \emph{$D$-pseudocompact} 
if  every sequence  $(O_i) _{i \in I} $ 
of nonempty open sets of $X$ has a $D$-limit point.
 Clearly,
we can equivalently 
relax the assumption that \emph{all} the $O_i$'s are nonempty
to the assumption that 
$ \{ i \in I \mid   O_i  \not= \emptyset    \} \in D $,
since we can change those $O_i$'s  
  in a set not in $D$ without affecting $D$-limit points of the sequence
(since $D$ is a filter).

Consider the following statement.

(*) $A$ and $B$ are open subsets of a GO space $X$, 
 $A \cup B = X$, and   
$a < b$, for every  $a \in A$ and $b \in B$.

An ordered pair $(A, B)$ satisfying (*) is called a \emph{gap} of $X$ 
in case that neither $A$ has a maximum, nor $B$ has a minimum.
Here we are including \emph{end gaps} in the definition of a gap,
that is, we allow either $A$ or $B$ to be empty. 
An ordered pair $(A, B)$ satisfying (*) is called a \emph{pseudo-gap}
in case that both $A$ and $B$ are nonempty,
and either $A$ has a maximum and $B$ has no minimum,
or $A$ has no maximum and $B$ has a minimum.
Clearly pseudo-gaps can occur only in GO spaces which are not  LOTS,
 since we are asking that $A$ and $B$ are open.

For a property $\mathcal P$, the expression ``$X$ has no (pseudo-)gap
satisfying $\mathcal P$'' shall be used as an abbreviation for ``$X$ 
has neither a gap nor a pseudo-gap satisfying $\mathcal P$''.

\section{$D$-convergence of a given sequence} \labbel{sing}

\begin{proposition} \labbel{single}
Suppose that $X$ is a linearly ordered set,
$D$ is an ultrafilter over some set $I$, 
and $(x_i) _{i \in I} $ is a sequence of elements of $X$.
Set
$A= \{ x\in X  \mid   \{ i \in I \mid x < x_i \}\in D \}$,
and 
$B= \{ x\in X  \mid   \{ i \in I \mid  x_i < x \}\in D \}$.
Then the following statements hold.
  \begin{enumerate}    
\item
$x \in X \setminus (A \cup B) $ if and only if  
$\{ i \in I \mid  x_i = x \} \in D$.
\item 
$|X \setminus (A \cup B)  | \leq 1$. 
\item
If $ x \in X \setminus (A \cup B) $,
then 
$a < x$, for every $a \in A$, and
$x < b$, for every  $b \in B$.
\item
If $a \in A$ and $b \in B$, then 
 $ \{ i \in I \mid  x_i \in (a, b) \} $ belongs to $  D$, hence
it is not empty. 
\item
If $a \in A$ and $b \in B$, then 
$a < b$. 
\item
If $X= A \cup B$, then either $A$ has no maximum or $B$ 
has no minimum. 
\item
If $X$ is a GO space, then
$(x_i) _{i \in I} $ $D$-converges to $x \in X$ if and only if 
one of the following (mutually exclusive) conditions holds.
  \begin{enumerate} 
   \item 
$\{ i \in I \mid  x_i = x \} \in D$, or
\item 
$x$ is the maximum of $A$ and $x \in \overline{B}$, or 
\item
$x$ is the minimum of $B$ and $x \in \overline{A}$.
  \end{enumerate}
\item
If $X$ is a LOTS, then
$(x_i) _{i \in I} $ $D$-converges to $x \in X$ if and only if 
either
  \begin{enumerate} 
   \item 
$\{ i \in I \mid  x_i = x \} \in D$, or
\item 
$x$ is the maximum of $A$ and $A \cup B = X$, or 
\item
$x$ is the minimum of $B$ and $A \cup B = X$.
  \end{enumerate}
\item
If $X$ is a GO space, then
$(x_i) _{i \in I} $ $D$-converges to some $x \in X$ if and only if 
$(A,B)$ is neither a gap, nor a pseudo-gap of $X$. 

In particular, if $X$ is a LOTS, then
$(x_i) _{i \in I} $ $D$-converges in $X$  if and only if 
$(A,B)$ is not a gap. 
\item 
Suppose that $X$ is a GO space
and that $X=A \cup B$.  
For $i \in I$, put
\begin{equation*} 
O_i=\begin{cases}
\bigcup _{a \in A}  (x_i, a) &    \text{if $x_i \in A$},  \\
\bigcup _{b \in B}  (b, x_i) &    \text{if  $x_i \in B$}.\\
\end{cases}
\end{equation*}
Then
$(x_i) _{i \in I} $ $D$-converges to $x $ if and only if 
$x$ is a $D$-limit point of $(O_i) _{i \in I} $.
\end{enumerate} 
  \end{proposition}

 \begin{proof}
(1) Since $D$ is an ultrafilter,
it follows that, for each $x \in X$, one and exactly one of the following sets:
$ \{ i \in I \mid x < x_i \}$,
$\{ i \in I \mid  x_i < x \}$,
$\{ i \in I \mid  x_i = x \}$ belongs to $D$,
since they are disjoint and their union is $I$.
Hence, by the definitions,  $x\not\in A \cup B$
if and only if  the  last eventuality occurs.

(2) If  $x,y\not\in A \cup B$, then, by (1),
both $X=\{ i \in I \mid  x_i = x \} $  and $Y=\{ i \in I \mid  x_i = y \} $ 
are in $  D$.
Since $D$ is  a filter, $ X \cap Y \in D$,
and,  since $D$ is proper, $ X \cap Y  \not= \emptyset   $.
If $ i \in X \cap Y $, then $x = x_i = y$. 

(3) 
If $a \in A$ and 
$ x \not\in  (A \cup B) $,
then both
$ \{ i \in I \mid a < x_i \}$ and  
$\{ i \in I \mid  x_i = x \}$ belong to $D$, by (1).
As above, their intersection is not empty,
and this implies $a < x$. 
In the same way, we get that 
$x < b$, for every  $b \in B$.

(4) If $a \in A$ and $b \in B$, then
$ \{ i \in I \mid a < x_i \}$ and  
$\{ i \in I \mid  x_i < b \}$ belong to $D$.
The intersection of the above two sets is 
$ \{ i \in I \mid  x_i \in (a, b) \}$,
therefore this set belongs to $D$, hence is nonempty.

(5) is immediate from (4).

(6) By contradiction, if $x$ is the maximum of $A$ and  $y$ is the minimum
of $B$, then, by (5) and since $X=A \cup B$, $y$ is the immediate successor of
$x$. Both
$\{ i \in I \mid x< x_i \}$ 
and
$\{ i \in I \mid x_i< y \}$ 
belong to $D$, since $x \in A$ 
and $y \in B$. 
But the above two sets are disjoint, a contradiction.

(7)  Suppose that $(x_i) _{i \in I} $ $D$-converges to $x \in X$.
By (1), (2) and the uniqueness of $D$-limits in Hausdorff 
spaces, either $ x \not\in  A \cup B$
and (a) holds, or
 $X= A \cup B$.
In this latter case, either $x \in A$ or $x \in B$.
 
Suppose that $ x \in A$. 
If $x$ is not the maximum of $A$, there is $x' \in A$
such that $x < x'$.
Then $( -\infty, x')$ is a neighborhood of 
$x$, and     
$\{ i \in I \mid x' < x_i\} \in D$,
since $x' \in A$. But then
$\{ i \in I \mid x_i \in ( -\infty, x') \} \not\in D$,
contradicting $D$-convergence.
Hence $x$ is the maximum of $A$. 
If $x \not\in \overline{B}$, then,
by (5) and since 
 $A \cup B=X$,
$(- \infty, x]$ is a neighborhood of $x$, 
hence, by $D$-convergence, 
$\{ i \in I \mid x_i \leq x\} \in D$, 
but this contradicts $x\in A$.
Next, notice that if $x \in A \cap \overline{B}$,
then necessarily $x$ is  the maximum of $A$,
because of (5). Hence $x \in \overline{B}$.

Symmetrically, if $x \in B$, 
then  $x$ is the minimum of $B$ and $x \in \overline{A}$.
Hence one among the conditions (a)-(c) holds.

Conversely,
if (a) holds, then trivially
$(x_i) _{i \in I} $ $D$-converges to $x$.

Suppose that (b) holds.
Since $x \in A$, then, by (4), 
$ \{ i \in I \mid x < x_i < b\} \in D$,
for every $b \in B$.
Since 
$x \in \overline{B}$.
then, for every neighborhood $U$ of $x$, 
there is  some $b \in B$ such that 
$U \supseteq [x,b)$,
but then
$ \{ i \in I \mid x_i \in U\}
\supseteq \{ i \in I \mid x < x_i < b\}
 \in D$,
thus 
$(x_i) _{i \in I} $ $D$-converges to $x$.

Case (c) is proved in a symmetrical way.

Notice that (b) and (c) are mutually exclusive, since 
$A$ and $B$ are disjoint, by (5). Conditions (b) and (c) are also both mutually
exclusive with (a)
 by  (1).

(8) Since every LOTS is, in particular, a GO space, then Condition (7) applies.
Hence it is enough to prove that, say, in every LOTS (8)(b) implies (7)(b). 
Indeed, if (8)(b) holds, then, 
since $A \cup B=X$, by (4) and (6), $x$  is an infimum of $B$,
but in a LOTS this implies $x \in \overline{B}$, thus (7)(b) holds
(notice that, by the very definition of $A$, if $A$ has a maximum, then 
$X \setminus A=B$
is not empty).

(9) We use the characterization given in (7). If one of (a), (b) or (c) in (7)
holds, then trivially $(A,B)$ is neither a gap nor a pseudo-gap. 

Conversely, suppose that $(A,B)$ is not a  (pseudo-)gap.
If $A \cup B \not = X$, then (7)(a) holds,
for some $x \in X$, by (1).
Otherwise, $A \cup B = X$.
Since we are including end gaps in our definition of a gap,
both $A$ and $B$  are nonempty. By (5), and since 
$A \cup B = X$, if neither $A$ has a maximum, nor $B$ 
has a minimum, then both $A$ and $B$ are open, contradicting the assumption
that $(A, B)$ is not a gap. Assume that, say, $A$ has a maximum $x$.
Since $A \cup B = X$, then $x$ has no immediate successor, thus $B$ has
no minimum, since $x$ is the maximum of $A$.
If  $x \not\in \overline{B}$, then $B$ would be clopen,
and $(A, B)$ would be a pseudo-gap.
Thus  $x \in \overline{B}$
and 6(b) holds.

(10) Suppose that $(x_i) _{i \in I} $ $D$-converges to $x $.
By (1), we are either in case (7)(b) or (7)(c). Suppose, say, that we
are in case (7)(b). Since $x \in \overline{B}$,  
then every neighborhood $U$ of $x$ 
contains an interval $[x,b']$, for some $b' \in B$.
For every $i \in I$ such that $x_i \in B$, 
there is $b \in B$  such that $b > \sup \{ b', x_i \} $, by (6). 
Thus
$b' \in [x, b'] \cap (b, x_i) \subseteq U \cap O_i$,
hence 
$U \cap O_i  \not= \emptyset   $.
Since $x$ is the maximum of $A$,
then, by (4), $\{ i \in I \mid x_i \in B \} 
= \{ i \in I \mid x < x_i  \} \in D$,
thus $x$ is a $D$-limit point of $(O_i) _{i \in I} $.

Conversely, suppose that 
$X= A \cup B$, and  that
$x$ is a $D$-limit point of $(O_i) _{i \in I} $.
Since $X= A \cup B$, then either 
$\{ i \in I \mid x_i \in A \} \in D$ or 
$\{ i \in I \mid x_i \in B \} \in D$, say, the latter eventuality occurs.
We first show that $x \not \in B$.
Indeed, if $x \in B$,
then
$\{ i \in I \mid x_i < x \} \in D$ and, since
$\{ i \in I \mid x_i \in B \} \in D$, then
$\{ i \in I \mid x_i \in B \cap (- \infty, x) \} \in D$.
In particular, the above set is not empty, hence there is
$x' < x$ such that $x' \in B$. 
Arguing in the same way,
$\{ i \in I \mid x_i \in B \cap (- \infty, x') \} \in D$.
But then, letting $U= (x', \infty)$, we have that $U$ is a neighborhood 
of $x$ such that $\{ i \in I \mid U \cap O_i = \emptyset \}\in D$,
since   $ O_i \subseteq (- \infty, x')  $,
whenever $x_i \in B \cap (- \infty, x')$.
This contradicts the assumption that
$x$ is a $D$-limit point of $(O_i) _{i \in I} $,
hence 
$x \not \in B$.
Thus $x \in A$. Moreover, $x$ is the maximum of $A$.
If not, there is some $x' \in A$ such that  $x< x'$.
Then $U=(-\infty, x')$ is a neighborhood of $x$ such that
$U \cap O_i = \emptyset$, for a set of indices in $D$,
namely, for    $\{ i \in I \mid x_i \in B \}$, by (5), again contradicting
the assumption that
$x$ is a $D$-limit point of $(O_i) _{i \in I} $.
In view of (7)(b), in order to finish the proof, 
it is enough to show that 
$x \in \overline{B}$. If not, $U=(-\infty, x]$ is a neighborhood of
$x$, and we can get a contradiction arguing as before,
since by (5) $A$ and $B$ are disjoint,
$A=U=(-\infty, x]$,
and $ O_i \subseteq B$, whenever $x_i \in B$.  
 \end{proof}

\begin{remark} \labbel{iff}   
From Proposition \ref{single}(10) it follows easily
that, for every ultrafilter $D$ and every GO space $X$, $D$-compactness 
of $X$ is equivalent
to $D$-\brfrt pseudocompactness of $X$.     
The only-if part is trivial. 
For the converse,
suppose that $X$ is $D$-\brfrt pseudocompact, and that 
$(x_i) _{i \in I} $ is a sequence of elements of $X$.
If $X \not = A \cup B$, then 
 $(x_i) _{i \in I} $ $D$-converges, by \ref{single}(1).
Otherwise,
define $(O_i) _{i \in I} $ as in (10).
It is enough to show that
$ \{ i \in I \mid   O_i  \not= \emptyset    \} \in D $,
since then we can  apply $D$-pseudocompactness to 
$(O_i) _{i \in I} $ and use (10).
Since $X= A \cup B$,
then either  
 $\{ i \in I \mid x_i \in A \} \in D$, or 
 $\{ i \in I \mid x_i \in B \} \in D$.
Say the latter case occurs, then,
by the definition,
$B$ has no minimum, hence,
 for every $i \in I$ such that $x_i \in B$,
 we have $O_i= B \cap (- \infty, x_i) \not= \emptyset $. 
Hence the $O_i$'s are nonempty for a set of indices in $D$,
namely, a set containing $\{ i \in I \mid x_i \in B \}$.  
\end{remark} 

However, much more can be said about 
the connections among $D$-compactness, $D$-pseudocompactness 
and other properties of a GO space $X$.
This is the main theme of the next section.

\section{$D$-compactness, $D$-pseudocompactness, gaps, products, etc.} \labbel{secth} 

To state the results of the present section
in their full generality, we need some more definitions.

An ultrafilter $D$ over some set $I$ is 
 \emph{$\lambda$-descendingly complete} if  
every 
$ \subseteq $-decreasing sequence
$(Z _ \alpha ) _{ \alpha \in \lambda } $ 
of sets in $D$ has intersection still in $D$.
The ultrafilter $D$ is \emph{$\lambda$-decomposable}
if there is a function $f:I \to \lambda $
such that $f ^{-1}(X) \not \in D $, for every 
$X \subseteq \lambda $ such that $|X| < \lambda $.
Clearly, if $\lambda$ is an infinite cardinal,     
a $\lambda$-decomposable ultrafilter is not $\lambda$-descendingly complete:
just consider $Z _ \alpha = f ^{-1}([ \alpha , \lambda )) $.
If $\lambda$ is regular, the converse holds.
If $\lambda$ is regular and  $D$  is not $\lambda$-descendingly complete,
as witnessed by $(Z _ \alpha ) _{ \alpha \in \lambda } $,
the following function $f$ witnesses that 
$D$ is $\lambda$-decomposable. Define $f$  by letting 
$f(i)$ be the smallest  $\alpha$ such that 
$i \not\in Z _ \alpha $, if such an $\alpha< \lambda $ exists;
 $f(i)$ is unimportant and can be arbitrary if 
$ i \in \bigcap _{ \alpha \in \lambda } Z _ \alpha $,
since we are assuming that $ \bigcap _{ \alpha \in \lambda } Z _ \alpha \not \in D$. 
Thus, for infinite regular cardinals, 
$\lambda$-decomposability is equivalent to the negation
of $\lambda$-descending completeness. 
Let us also mention  that
there is another notion equivalent to $\lambda$-\brfrt decomposability, 
for $\lambda$ an infinite regular cardinal,
that is, \emph{$ ( \lambda , \lambda )$-regularity}. We shall
make only a limited use of $ ( \lambda , \lambda )$-regularity here
(see Sections \ref{secinit} and \ref{setth}), but we should warn the reader that
 theorems about 
$\lambda$-descending (in)completeness 
or $\lambda$-decomposability are 
frequently stated in equivalent
forms in terms of regularity, or vice versa.
A full discussion is given in \cite{mru};
see Section 1 there and, in particular, Properties 1.1(xi) and Consequence 1.2.

For an ultrafilter $D$, let 
$K_D= \{ \nu \geq \omega  \mid D \text{ is $\nu$-decomposable}\} $.
Many problems are still open about the possible values
 $K_D$  can assume,
and solutions to such problems are heavily dependent on the 
universe of set theory one is working in. 
For example, in certain models of set theory, $K_D$
is always an interval of cardinals, with bottom element $ \omega$;
on the other hand, there are models in which
$K_D$ is a rather sparse set of cardinals.
We refer
to the comments after Problem 6.8 in \cite{mru}  
for further information and details,
or to the remarks before Corollary \ref{gocorbbb} here, 
where we
 show that such set theoretical problems affect the behavior 
under
products
of 
certain compactness properties of GO spaces. In particular, certain statements turn out
to be independent from the usual axioms for set theory.
In this paper we shall be mainly concerned with \emph{regular}
cardinals in $K_D$, hence we shall establish the special notation
$K_D ^{Reg}$ for the set of such cardinals. 
Namely, 
$K_D^{Reg}= \{ \nu \geq \omega  \mid \nu \text{ is regular and } D \text{ is $\nu$-decomposable}\} $.

A topological space is \emph{$[ \nu, \lambda ]$-compact}
if every open cover of cardinality $\leq \lambda $ has a subcover of cardinality
$<\nu$. It is a classical result 
by Alexandroff and Urysohn \cite{AU} that
if $\nu$ is a regular cardinal, 
then $[ \nu, \nu ]$-compactness of a topological space $X$ 
is equivalent 
to $\CAP_ \nu$, which is the property asserting
that every subset  of $X$ of cardinality $\nu$
has a complete accumulation point. 
 \emph{Initial $\lambda$-compactness}
is $[ \omega , \lambda ]$-compactness, 
and, again by \cite{AU},
it is equivalent to 
$[ \nu, \nu ]$-compactness,
for every cardinal $\nu \leq \lambda $, 
equivalently, for every \emph{regular} cardinal $\nu \leq \lambda$. 
As usual, by  $[ \nu, \lambda ]$ we shall also denote the \emph{interval}
of all cardinals $\mu$ such that $\nu\leq \mu \leq \lambda$.
We hope that the partially overlapping notation will cause no confusion. 

A topological space is \emph{weakly $[ \nu, \lambda ]$-compact}
if every open cover of cardinality $\leq \lambda $ has a subfamily of cardinality
$<\nu$ with dense union.
\emph{Weak initial $\lambda$-compactness} 
is  weak $[ \omega , \lambda ]$-compactness.
The above notions are frequently named using very
disparate terminology. 
Weak $[ \nu, \lambda ]$-compactness 
has sometimes been called 
\emph{weak $\lambda$-$\nu$ compactness},
and we used still different terminology in 
\cite{tproc2}, where we called it
$\mathcal O$-$[ \nu, \lambda ]$-compactness.
When $\nu$ is a regular cardinal,
weak $[ \nu, \nu ]$-compactness is equivalent
to a notion called \emph{pseudo-$( \kappa , \kappa  )$-compactness}. 
Weak initial $ \omega $-compactness
is equivalent to  \emph{faint compactness}, 
and equivalent to pseudocompactness in the class of Tychonoff spaces.
Weak initial $\lambda$-compactness has been called
\emph{almost $\lambda$-compactness}, or  
\emph{weak-$\lambda$-$\aleph_0$-compactness} by some authors.
See \cite[Remark 3]{tapp2} for further details and references.

If $ \lambda $ is an infinite cardinal,  a sequence
 $(x_ \gamma ) _{ \gamma < \lambda  } $ of elements of a topological space 
\emph{converges} to some point $x$   if,
for every neighborhood $U$ of $x$,
there is $\gamma < \lambda  $ such that 
$x _{ \gamma '} \in U $, for every $\gamma' > \gamma $.   

If $(A,B)$ is a gap of some GO space,
we say that that \emph{$\lambda$ is a type of}   $(A,B)$
if either $A$ has cofinality $\lambda$, or $B$ has
coinitiality $\lambda$ (thus a gap has at most two types).
 If $(A,B)$ is a pseudo-gap, say, $B$ having a minimum,
the \emph{type of} $(A,B)$ is the cofinality of $A$,
and, symmetrically, the coinitiality of $B$, if $A$ has a maximum.

The above two notions are related as follows.
If $(x_ \gamma ) _{ \gamma < \lambda  } $ 
is a non converging strictly increasing sequence
of elements of a GO  space $X$,
then, letting
$A = \{ x \in X \mid x < x_ \gamma \text{ for some  } \gamma < \lambda  \}$,
and $B=X \setminus A$,  
we have that $(A,B)$ is a (pseudo-)gap
having type $\lambda$.
A symmetrical statement holds for 
strictly decreasing sequences.
Conversely,
given 
 a (pseudo-)gap
$(A,B)$ having type $\lambda$,
we obtain
a strictly monotone non converging sequence
$(x_ \gamma ) _{ \gamma < \lambda  } $,
by considering either a cofinal subset of $A$ or
a coinitial subset of $B$.

\begin{theorem} \labbel{thm}
Suppose that $D$ is an ultrafilter over some set $I$,
and recall that $K_D ^{Reg}$ is the set of those infinite regular cardinals $\nu$ 
such that $D$ is $\nu$-decomposable.
For  every GO
space $X$, the following conditions are equivalent.
 \begin{enumerate}
\item 
$X$ is $D$-compact.
\item
$X$ is $[ \nu, \nu]$-compact, for every  cardinal $\nu \in K_D ^{Reg}$.
\item
For every   cardinal $ \nu \in K_D ^{Reg}$,
and every strictly increasing (resp., strictly decreasing) $\nu$-indexed sequence
of elements of $X$, the sequence has a supremum (resp., an infimum) to which it
converges.
\item
$X$ has no (pseudo-)gap of type belonging to $K_D ^{Reg}$.  
\item 
$X$ is $D$-pseudocompact.
\item
$X$ is weakly $[ \nu, \nu]$-compact, for every   cardinal $\nu \in K_D ^{Reg}$.
\item
For every cardinal $ \nu \in K_D  ^{Reg} $,
and every sequence  $(O_ \gamma ) _{\gamma < \nu}$ of  open 
 nonempty sets of $X$,   
 there is $x \in X$ such that 
$|\{\gamma < \nu \mid O_ \gamma  \cap U \not= \emptyset \}|= \nu$, 
for every neighborhood $U$ of $x$.   

In the  above condition we can equivalently require either that the $O_ \gamma $'s
are pairwise disjoint, or  that $O_ \gamma \subset O _{ \gamma '} $,
for $\gamma> \gamma '$.  
\item
$X$ is $D'$-compact,
for every ultrafilter $D'$ such that 
$K _{D'}  ^{Reg} \subseteq K_D ^{Reg}$.
\item
$X$ is $D'$-pseudocompact,
for every ultrafilter $D'$ such that 
$K _{D'}^{Reg} \subseteq K_D ^{Reg}$.
  \end{enumerate} 
 \end{theorem}

\begin{proof}
(1) $\Rightarrow $  (2) It is enough to show that if 
$\nu$ is regular and $D$ is $\nu$-decomposable, 
then every $D$-compact topological space is $[ \nu, \nu]$-compact.
This is essentially a classical result (which holds for every topological space),
due to \cite{GS} in the case $\nu= \omega $, and to 
\cite{sak78} in the general case.
See also \cite{Ca} and \cite{tproc2} for other versions and generalizations.
Briefly,  and using the equivalent reformulation of
$[ \nu, \nu]$-compactness in terms of complete accumulation points, 
if $Y \subseteq X$ and $|Y |= \nu$,
enumerate $Y$ as $\{ y _ \gamma \mid \gamma \in \nu\}$.
If $f:I \to \nu$ witnesses the $\nu$-decomposability of $D$,
then, by $D$-compactness, the sequence $(y _{f(i)} ) _{i \in I} $  
has some $D$-limit point, which is easily seen to be also a complete accumulation
point of $Y$.

(2) $\Rightarrow $  (3) is almost immediate.
Indeed, if, by contradiction, $(x_ \gamma ) _{ \gamma < \nu} $ is, say,
strictly increasing, then  
$ O = X \setminus \bigcup _{\gamma < \nu} (- \infty, x_ \gamma ) $
is open. Then the family containing $ O$  and
$(- \infty, x_ \gamma ) $, for $\gamma < \nu$,
is an open cover by $\nu$ sets, but no subfamily by $<\nu$ sets is a cover,
since $\nu$ is regular.  

The equivalence of (3) and (4) should be clear from the remarks shortly before
the statement of the theorem.

The proof of (4) $\Rightarrow $  (1) is the key argument in the proof
of the theorem, and uses in an essential way the assumption that
we are in a GO space.
Suppose by contradiction that (1) fails
and let $(x_i) _{i \in I} $ be a sequence which does not $D$-converge.
Let $A$ and $B$ be defined as in the statement of Proposition \ref{single}.
By condition (9) in the same proposition, $(A,B)$ is either a gap or a pseudo-gap
of $X$, in particular, $A \cup B =X$. Hence either
$\{i \in I \mid x_i \in A\} \in D$,  or $\{i \in I \mid x_i \in B\} \in D$.
Suppose, say, that the latter occurs. Then
by the assumption that $\{i \in I \mid x_i \in B\} \in D$, by the definition of $B$,
and since $D$ is a filter, we get that, for every $b \in B$, 
$X_b=\{i \in I \mid x_i \in (-\infty, b) \cap B\} \in D$.
Notice that it follows that $B$ has no minimum,
since each $X_b$ is nonempty.
Choose a strictly decreasing coinitial sequence $(b_ \gamma ) _{ \gamma \in \nu} $ in $B$.
Then $\nu$ is a type of the (pseudo-)gap $(A,B)$,
and $ (X _{b_ \gamma }) _{ \gamma \in \nu} $ 
witnesses that $D$ is not $\nu$-descendingly complete,
since 
$ \bigcap  _{ \gamma \in \nu} X _{b_ \gamma }=
\{i \in I \mid x_i \not\in B\} \not\in D $. 
Thus $\nu\in K_D ^{Reg}$, by the equivalence mentioned after the definition 
of decomposability, and since $\nu$ is regular,
being the coinitiality of $B$. Thus $(A,B)$  has type $\nu \in K_D ^{Reg}$, 
contradicting (4). 

So far, we have proved the equivalence of (1)-(4).

(1) $\Rightarrow $  (5) is trivial.

(5) $\Rightarrow $  (6) is similar to (1) $\Rightarrow $  (2),
by considering a sequence of nonempty open sets of $X$, rather than a sequence 
of elements of $X$.
Full details can be found in \cite[Fact 6.1, Corollary 4.6 and condition (d) in Theorem 4.4]{tproc2}, by taking $\mathcal F = \mathcal O$ there. 

The equivalence of (6) with the first statement in  (7) is true in every topological space,
and is similar to the equivalence of $[ \nu, \nu]$-compactness with
$\CAP_ \nu$ (recall that we 
are assuming that $\nu$ is regular). Details can be found again in 
\cite{tproc2}, by taking $\mathcal F = \mathcal O$
in Theorem 4.4(a) $\Leftrightarrow $  (c) there.

For GO spaces, (7) $\Rightarrow $  (3) is a standard argument. 
Say, $(x_ \gamma ) _{ \gamma \in \nu} $ is strictly increasing.
Put $O_ \gamma = \bigcup _{ \eta > \gamma } (x_ \gamma , x_ \eta) $.
Then if $x$ is such that  
$|\{\gamma < \nu \mid O_ \gamma  \cap U \not= \emptyset \}|= \nu$, 
for every neighborhood $U$ of $x$, then necessarily
$(x_ \gamma ) _{ \gamma \in \nu} $ converges to $x$.
Here the $O_ \gamma $'s are strictly decreasing with respect to inclusion.

If we want the $O_ \gamma $'s  to be pairwise disjoint,
then, for $\gamma= \alpha + n$ with $\alpha=0$ or $\alpha$ limit,
take   $O_ \gamma = (x _{ \alpha + 2n}, x  _{ \alpha + 2n+2}  )$.

Hence (1)-(7) are all equivalent, for every ultrafilter $D$ and every GO space $X$.

(2) $\Rightarrow $  (8)  Suppose that $D'$ is such that 
$K _{D'}^{Reg} \subseteq K_D^{Reg}$.
If (2) holds, then, since $K _{D'}^{Reg} \subseteq K_D^{Reg}$,
$X$ is $[ \nu, \nu]$-compact, for every   cardinal $\nu \in K _{D'}^{Reg}$.
Applying the equivalence of (1) and (2) \emph{in the case of the ultrafilter $D'$},
we get that $X$ is $D'$-compact. 

(8) $\Rightarrow $  (1) follows trivially by taking $D=D'$.

The equivalence of (9) with, say, (6) and (5)  is similar.
\end{proof}  

Given a family $(X_j) _{j \in J} $ of 
GO
spaces, the product 
$ \prod _{j \in J} X_j$ is not necessarily a GO space;
however, it is a topological space, when endowed
with the (Tychonoff) product topology.
$ \prod _{j \in J} X_j$ can also be given the structure of a 
partially ordered set, by letting the relation $ \mathbf{x}\leq \mathbf{y}$ 
hold in $ \prod _{j \in J} X_j$ if and only if it holds componentwise. 
The next corollary
deals with the above structures on $ \prod _{j \in J} X_j$.
It shows that if each $X_j$ is a GO space, then many properties of $ \prod _{j \in J} X_j$
are determined by the corresponding properties of
the $X_j$'s.  

To avoid trivial exceptions, we shall always assume that
all factors in a product are nonempty.

\begin{corollary} \labbel{corthm}
Suppose that $D$ is an ultrafilter over some set $I$.
For  every family
$(X_j) _{j \in J} $ of 
GO
spaces, the following conditions are equivalent.
  \begin{enumerate}    
\item[(a)] 
For every $j \in J$, the GO space $X_j$ satisfies one (and hence all) of the conditions
in Theorem \ref{thm}.   
\item[(b)] 
The topological space  $ \prod _{j \in J} X_j$ satisfies all of the conditions
(1)-(2),(5)-(9) in Theorem \ref{thm}.  
\item[(c)] 
The topological space  $ \prod _{j \in J} X_j$ satisfies one of the conditions
(1)-(2),(5)-(9) in Theorem \ref{thm}.  
\item[(d)] For every   cardinal $\nu \in K_D ^{Reg}$, every monotone $\nu$-indexed sequence in   $ \prod _{j \in J} X_j$ converges.
  \end{enumerate} 
 \end{corollary}

 \begin{proof}
(a) $\Rightarrow $    (b) If each 
$X_j$ satisfies any one of the conditions
in Theorem \ref{thm}, then, by the very same theorem,
$X_j$ is $D$-compact.
By an easy and classical property of $D$-compactness
\cite{ber70,GS,sak78}, every product of $D$-compact spaces is still 
$D$-compact, hence 
$ \prod _{j \in J} X_j$ is still $D$-compact.
Now notice that the proof that (1)
 implies any one of the conditions
(1)-(2),(5)-(7) in Theorem \ref{thm}
holds for an arbitrary topological space, not only for GO spaces.
Moreover, if each $X_j$ is $D$-compact, then, by Theorem \ref{thm} (1) $\Rightarrow $ (8),
each $X_j$ is $D'$-compact, for every 
ultrafilter $D'$ such that 
$K _{D'} ^{Reg} \subseteq K_D ^{Reg}$.
Again since $D'$-compactness is preserved under products,
$ \prod _{j \in J} X_j$ is  $D'$-compact, thus condition
\ref{thm} (8) holds for $ \prod _{j \in J} X_j$.
The proof that (9) holds for $ \prod _{j \in J} X_j$
is similar, since $D'$-pseudocompactness is preserved under products,
too, \cite{GS}. 

(b) $\Rightarrow $  (c) is trivial.

(c) $\Rightarrow $  (a) 
If $ \prod _{j \in J} X_j$ satisfies any one of the conditions
(1)-(2),(5)-(9), then trivially
each $X_j$ satisfies the same condition, hence (a) holds.

(a) $\Rightarrow $  (d) By Theorem \ref{thm},
if (a) holds, then condition (3) in Theorem \ref{thm} holds, for every
$j \in J$.
 Suppose that $(\mathbf{x}_ \gamma ) _{ \gamma \in \nu} $ 
is a monotone, say, increasing, sequence in $ \prod _{j \in J} X_j$.
Since $\nu$ is a regular cardinal, 
for every $j \in J$, the projection 
$(x _{j,  \gamma} ) _{ \gamma \in \nu} $
of  $(\mathbf{x}_ \gamma ) _{ \gamma \in \nu} $
into $X_j$ 
is either eventually constant, or 
has a strictly increasing subsequence of order type $\nu$.
In both cases, 
$(x _{j,  \gamma} ) _{ \gamma \in \nu} $ converges in
$X_j$. This is trivial in the former case; 
in the latter case, the strictly increasing subsequence
converges, by condition (3) in Theorem \ref{thm}, and then also
$(x _{j,  \gamma} ) _{ \gamma \in \nu} $ converges (to the same point), since it is increasing.
Now it is trivial to see that a sequence in a product converges if and only if
each projection converges, hence 
$(\mathbf{x}_ \gamma ) _{ \gamma \in \nu}$  converges in 
$ \prod _{j \in J} X_j$.

(d) $\Rightarrow $  (a) If (d) holds, then trivially each $X_j$
satisfies condition (3) in Theorem \ref{thm}, thus (a) holds.  
 \end{proof}

\section{The particular case of initial compactness} \labbel{secinit} 

For LOTS, the equivalence of conditions (2), (4)  and (10) in the next corollary 
has first been proved, under different terminology and with further
equivalences, in \cite[Theorem 3]{GFW}. For GO spaces the corollary 
 appears in \cite{goic}. 

Recall that an ultrafilter $D$ over $\lambda$
is \emph{regular} if  there is a family 
$(Z_ \alpha ) _{ \alpha \in \lambda } $ of members of $D$ such that
the intersection of any infinite subfamily is empty (this is also
called \emph{$( \omega, \lambda )$-regularity}). 
It is a standard application of the Axiom of Choice, 
or just the Prime Ideal Theorem,
to show that, for every infinite cardinal $\lambda$, there exists a regular ultrafilter $D$ 
over $\lambda$. See, e.~g., Chang and Keisler \cite[Proposition 4.3.5]{CK}.
For such an ultrafilter, it is easy to show that $ \nu \in K_D$,
for every regular $\nu \leq \lambda $: see \cite[Properties 1.1(xii)]{mru}  
(in fact, one actually has that
$K_D=[ \omega, \lambda ]$,
by  \cite[Properties 1.1(1) and the last paragraph in Remark 1.5(b)]{mru}, but we shall not need this here).
It follows that $K _{D'} ^{Reg}  \subseteq K_D^{Reg}$, for every regular ultrafilter over $\lambda$,
and every ultrafilter $D'$ over any set of cardinality $\leq \lambda$.

\begin{corollary} \labbel{initial}
For every infinite cardinal $\lambda$, and every GO
space $X$, the following conditions are equivalent.
 \begin{enumerate}
\item
$X$ is $D$-compact, for some regular ultrafilter over $\lambda$.
   \item 
 $X$ is initially $\lambda$-compact.
\item
For every infinite regular cardinal $ \nu \leq \lambda $,
and every strictly increasing (resp., strictly decreasing) $\nu$-indexed sequence
of elements of $X$, the sequence has a supremum (resp., an infimum) to which it
converges.
\item
$X$ has no (pseudo-)gap having type $\leq \lambda$.
\item
$X$ is $D$-pseudocompact, for some regular ultrafilter over $\lambda$.
\item 
 $X$ is weakly initially $\lambda$-compact.
\item
For every infinite (equivalently, every infinite regular) cardinal $ \nu \leq \lambda $,
and every sequence  $(O_ \gamma ) _{\gamma < \nu}$ of open 
 nonempty sets of $X$,   
 there is $x \in X$ such that 
$|\{\gamma < \nu \mid O_ \gamma  \cap U \not= \emptyset \}|= \nu$, 
for every neighborhood $U$ of $x$.   

In the  above condition we can equivalently require either that the $O_ \gamma $'s
are pairwise disjoint, or  that $O_ \gamma \subset O _{ \gamma '} $,
for $\gamma> \gamma '$.  
\item 
$X$ is $D$-compact, for every ultrafilter $D$ over any set of cardinality 
$ \leq\lambda$.
\item 
$X$ is $D$-pseudocompact, for every ultrafilter $D$ over any set of cardinality 
$ \leq\lambda$.
\item
$X$ is \emph{$\lambda$-bounded}, that is, every subset of cardinality
$\leq \lambda$ has compact closure.
  \end{enumerate} 
 \end{corollary} 

\begin{proof}
We are first giving the proof as a consequence 
of Theorem \ref{thm}. As we mentioned in the introduction,
a more direct proof (which nevertheless uses essentially the same ideas)
can be found in  \cite{goic}.
We then sketch an alternative proof which goes along with some arguments in
\cite{GFW}.

Most of the equivalences in the corollary 
are the particular cases of Theorem \ref{thm} applied when
$D$
is a regular ultrafilter over $\lambda$.
This is the case for conditions (1), (3)-(5), (7)$_{\rm reg}$,
(8) and (9),
by the remarks before the statement of the corollary,
and where by (7)$_{\rm reg}$ we denote condition (7)
restricted to regular $\nu$'s.

As far as (2) is concerned, 
it is standard and already proved in \cite{AU}
that initial $\lambda$-compactness is equivalent
to $[\nu,\nu]$-compactness for every regular $\nu \leq \lambda $,
so we get the equivalence with condition (2) in Theorem \ref{thm},
in the case when $D$ is regular over $\lambda$.
Alternatively, (1) $\Rightarrow $  (2) in the present corollary 
follows immediately from \cite[Theorem 3.4]{Ca}.
Notice that \cite{Ca}, following standard use
in the model theoretical setting, uses a notation 
in which the order of the cardinals is reversed, 
both in the definition of $[\mu,\nu]$-compactness
and of $( \omega , \nu)$-regularity. 

In contrast with (2), it is not always necessarily the
case that (for  topological spaces in general) weak $[\nu,\nu]$-compactness for every regular $\nu \leq \lambda $
implies weak initial $\lambda$-compactness. This is
Remark 30 in \cite{tapp2}, 
relying on an example by Garc{\'{\i}}a-Ferreira 
\cite{gf99}, which in turn builds  on a construction by 
Kanamori \cite{kan86}.
However, 
(5) implies (6) by \cite[Corollary 15]{tapp2},
(6) implies  (7)  by 
\cite[Theorem 1]{mpcap} and Retta \cite[Theorem 3(d)]{ret93},
and (7) implies (7)$_{\rm reg}$ trivially,
hence, for GO spaces, they are all equivalent, since  we have already
proved that for GO spaces (5) and (7)$_{\rm reg}$ are equivalent.

Finally, the equivalence of (8) and (10)  holds for every Hausdorff regular space, \cite[Theorems 5.3 and 5.4]{sak78} (recall that it can be proved that
every GO space is regular). Our first proof of the corollary is thus complete.

\smallskip

An alternative proof of (4) implies (10)
 goes as follows. Suppose that $Y \subseteq X$, and  
$|Y| \leq \lambda $.  The closure $ \overline{Y} $ of $Y$ has no (pseudo-)gap
of type $\leq \lambda$, since otherwise it would extend to a (pseudo-)gap of $X$
having the same type. Moreover, $ \overline{Y} $  has no
subset of order type or reversed order type $>\lambda$,
since then one could construct a subset of $Y$ having the same order type,
but this is impossible, since $|Y| \leq \lambda $. In conclusion,
$ \overline{Y} $ has no (pseudo-)gap at all,
but this implies that it is compact, by a well-known result,
e.~g., Nagata \cite[Theorem VIII.2]{nag85}. 

Then an alternative proof of the corollary is obtained by the following 
chains of implications
(4) $\Rightarrow $   (10) $\Rightarrow $  (8) $\Rightarrow $  (1)
$\Rightarrow $  (2) $\Rightarrow $  (3) $\Rightarrow $  (4)
and (4) $\Rightarrow $   (10) $\Rightarrow $ (8) $\Rightarrow $   (9) $\Rightarrow $  (5)
$\Rightarrow $  (6) $\Rightarrow $  (7) $\Rightarrow $  (4),
which are either trivial, or proved before.
\end{proof}

\begin{corollary} \labbel{cor}
Suppose that $X$ is a product of topological spaces and that all
 factors but at most one are GO spaces.
Then the following hold.
  \begin{enumerate}    \item  
$X$ is  initially $\lambda$-compact
if and only if each factor is initially $\lambda$-compact.
\item
$X$ is weakly initially $\lambda$-compact
if and only if each factor is weakly initially $\lambda$-compact.
 \end{enumerate}
 \end{corollary} 

\begin{proof}
(1) An implication is trivial.
For the other direction, by the equivalence of (2) and (10) in Corollary \ref{initial}, 
all but at most one factor are $\lambda$-bounded.
A product of regular $\lambda$-bounded spaces
is still $\lambda$-bounded 
\cite[Lemma 4]{GFW}, 
or \cite[Theorem 5.7 and implications (1), (1$'$) in Diagram 3.6]{ste84},
and  a product of a $\lambda$-bounded space with an 
  initially $\lambda$-compact space is  initially $\lambda$-compact
\cite[Theorem 5.2 and implications (1), (2) in Diagram 3.6]{ste84}. Hence 
 (1) follows by first grouping together
the GO spaces, and then, in case,  multiplying their product
with the possibly non GO factor.
Alternatively, one can use an argument parallel to the one 
we are going to give for (2).

(2)  Again, an implication is trivial.
For the other direction, by the equivalence of (6) and (9) in Corollary \ref{initial}, 
all but at most one factor are $D$-pseudocompact, for every ultrafilter
$D$ over any set of cardinality $\lambda$.
Since  $D$-pseudocompactness is preserved under products
\cite{GS},
we have that all but at most one factor are
$D$-pseudocompact, for every ultrafilter
$D$ over any set of cardinality $\lambda$.
Since the (possible) remaining factor is
weakly initially $\lambda$-compact, by assumption,
(2) follows from 
 the
next lemma (which, in some form or another, is probably folklore).
 \end{proof}

\begin{lemma} \labbel{lmlm} 
If the topological space $X$ is  $D$-pseudocompact, for every ultrafilter $D$ 
over $\lambda$,
and the topological space $Y$ is weakly initially $\lambda$-compact, then 
$X \times Y$ is weakly initially $\lambda$-compact.
\end{lemma}

 \begin{proof}
Let $S_ \omega ( \lambda )$ denote the set of all finite subsets of $\lambda$.
By \cite[Theorem 10]{tapp2}, a topological space is 
weakly initially $\lambda$-compact if and only if,
 for every sequence $\{ O_Z \mid Z \in S_ \omega ( \lambda ) \}$
of nonempty open sets,
 there exists an ultrafilter D over $S_ \omega ( \lambda ) $
such that  both
$\{ Z \in  S_ \omega ( \lambda ) | \alpha \in Z\} \in D$,
for every $\alpha \in \lambda $, and
$\{ O_Z \mid Z \in S_ \omega ( \lambda ) \}$ has a $D$-limit point.
 
 So let $\{ O_Z \mid Z \in S_ \omega ( \lambda ) \}$
be a sequence of nonempty open sets in $X \times Y$: 
it is enough to find some $D$ 
as above. 
Without loss of generality, 
we can suppose that  $O_Z= U_Z \times V_Z$, for every $Z \in S_ \omega ( \lambda )$. Since $Y$ is  
weakly initially $\lambda$-compact, by the quoted theorem,
 there exists an ultrafilter D as above
such that  
$\{ V_Z \mid Z \in S_ \omega ( \lambda ) \}$ has a $D$-limit point $y$.
Since $| S_ \omega ( \lambda )|= \lambda $,
then  $X$ is  $D$-pseudocompact, hence also
 $\{ U_Z \mid Z \in S_ \omega ( \lambda ) \}$ has a $D$-limit point $x$.
But then $(x,y)$ is a $D$-limit point of  
$\{ O_Z \mid Z \in S_ \omega ( \lambda ) \}$.
 \end{proof}

See \cite[Corollary on p. 203]{GFW} 
for other results related to Corollary \ref{cor}. 
See \cite{GFS}
for further results in the particular case of countable compactness, that is,
 $\lambda= \omega $.

\section{The impact of set theory} \labbel{setth} 

In view of Theorem \ref{thm} (2) $\Leftrightarrow $  (4), 
one would be tempted to conjecture that,
for every infinite regular cardinal $\nu$, a GO space 
$X$ is $[ \nu, \nu]$-compact
if and only if $X$ ha no (pseudo-)gap of type 
$\nu$. However this is false:
just consider $\mathbb{R}$ with the usual order, 
and with the discrete topology.
The counterexample can be easily
turned into a LOTS: let $X$ be the product $\mathbb{R} \times \mathbb{Z} $,
with the lexicographic order. 
The induced topology on $X$  
is the discrete one and, since
$| X| = 2 ^{ \omega } $,
then $X$ is not $[\nu, \nu]$-compact,
for every $\nu \leq   2 ^{ \omega } $.
On the other hand, 
$X$ has only gaps of type $ \omega$, thus, 
taking $\nu= \omega _1$ 
(or any $\nu$ with $ \omega <\nu \leq 2 ^{ \omega } $---of course,
this is significant only if the Continuum Hypothesis fails),
we get that having no gap of type $\nu$ does not 
necessarily imply
$[\nu, \nu]$-compactness.
Indeed, the equivalence holds
only for a very special class of cardinals.
It can be shown that a regular cardinal 
 $\nu > \omega $ is  \emph{weakly compact}
 if and only if, for every GO space 
(equivalently, LOTS)
$X$,   
$[ \nu, \nu]$-compactness of $X$ is equivalent 
to  $X$ having  no (pseudo-)gap of type 
$\nu$.

Put in another way, an equivalence like 
 (2) $\Leftrightarrow $  (4) in Theorem \ref{thm}
can hold only if considered \emph{simultaneously}
for all $\nu $ in some appropriate class $H$ of regular cardinals.
Theorem \ref{thm} shows that the equivalence holds 
 in case $H$ can be realized as $  K_D ^{Reg} $,
for some ultrafilter $D$.
In this sense, the existence of a regular ultrafilter 
over $\lambda$ can be seen as a fact 
 ``responsible behind the scene'' for the 
equivalence of initial $\lambda$-compactness 
with the nonexistence of (pseudo-)gaps of type $\leq\lambda$.
As we mentioned, we have a proof
that if $H= \{ \nu \} $ is a singleton,
then the equivalence \ref{thm}
(2) $\Leftrightarrow $  (4) (with $H$ in place of $  K_D ^{Reg} $ there)
holds if and only if either $\nu= \omega $,
or $\nu$ is weakly compact.

The above considerations suggest the following  definition.
If $H$ a class of infinite cardinals,
let us say that a topological space is \emph{$H$-compact}
if it is $[\nu, \nu]$-compact, for every $\nu \in H$.  
Thus the equivalence (2) $\Leftrightarrow $  (4) in Theorem \ref{thm}
asserts that if $H=K_D ^{Reg}$, for some ultrafilter $D$,
then a GO space $X$ is $H$-compact if and only if 
$X$ has no (pseudo-)gap of type in $H$.
If we strengthen  the above conditions
by asking forms of preservation under products,
we get that the corresponding equivalence holds \emph{exactly} for those classes
$H$ which can be expressed as unions of classes of the
form $K_D ^{Reg}$. 
This is the content of  the next corollary, which also shows that 
in many cases, say, when
$H$ is finite, the equivalence holds if and only if 
$H=K_D ^{Reg}$, for some ultrafilter $D$.
Thus, for GO spaces,  preservation of $H$-compactness under products
is highly dependent on set theory, as we shall discuss shortly after the 
proof of the corollary.

Recall that an ultrafilter $D$ is \emph{$( \lambda ,\mu ) $-regular} 
if there is a set of $\mu $ members of $D$ such that the intersection of 
any $\lambda$  of them is empty. We refer again to \cite{mru} for more information about
regularity of ultrafilters. Here regularity shall be used in connection
with results from \cite{Ca}, asserting that, roughly,  $( \lambda ,\mu ) $-regular
ultrafilters are standard witnesses for $[ \lambda , \mu]$-compactness
of products of topological spaces.

Let $Reg$ denote the class of all infinite regular cardinals. 
An \emph{interval of regular cardinals} is a set of the form
$Reg \cap [ \lambda, \mu]$, for certain cardinals $\lambda$ and $\mu $.
Notice that we allow $\mu $ to be singular.
If $H$ is a set of cardinals, and $F$ is a finite union of   intervals of regular cardinals, 
we say that \emph{$H$ includes the cofinalities of the extremes of  $F$} if $F$ can 
be represented as
$F = Reg \cap  \bigcup _{p \in P} [ \lambda _p, \mu _p] $, with $P$ finite,
and in such a way that $ \cf \mu_p \in H$, for every $p \in P$
(of course, in case $H=F$ this is relevant only when some $\mu _p$ is singular).

\begin{corollary} \labbel{gocorbb}
Suppose that $H$ is a class of infinite regular cardinals. Then the following statements are equivalent. 
\begin{enumerate}
   \item 
Every  product of a family of $H$-compact GO spaces is still $H$-compact.
\item
Same as (1), restricted to spaces  which are  regular cardinals
with the order topology.
\item
For every $\nu \in H$ there is an ultrafilter $D_\nu$ 
such that $\nu \in K _{D_\nu} $ and $K _{D_\nu} ^{Reg} \subseteq H$. 
\item[(3$'$)]
There is a class $\mathcal D$ of ultrafilters such that
$H= \bigcup _{D \in \mathcal D} K _{D} ^{Reg} $. 
\item
Every product of GO spaces without (pseudo-)gaps of type in $H$ is
$H$-compact.
\item
More generally, if  $Y$ is
a product of GO spaces,
and each factor of $Y$ satisfies
at least one of conditions (2)-(4), (6)-(7)
in Theorem \ref{thm}, with $K_D^{Reg}$ there replaced 
by $H$, then $Y$  is
$H$-compact.
\item[(5$'$)]
If  $Y$ is
a product of GO spaces,
and each factor of $Y$ satisfies
all the conditions (2)-(4), (6)-(7)
in Theorem \ref{thm}, with $K_D^{Reg}$ there replaced 
by $H$, then $Y$  is
$H$-compact.
\item
For every finite $F \subseteq  H$, there is an ultrafilter $D_F$ 
such that $F \subseteq K _{D_F} ^{Reg} \subseteq H$. 
\item
More generally, for every 
set $F \subseteq H$ such that $F$ is a finite union of
intervals of regular cardinals, and $H$ includes the cofinalities of the extremes of $F$,
 there is an ultrafilter $D_F$ 
such that $F \subseteq K _{D_F} ^{Reg}\subseteq H$. 
\end{enumerate}
Suppose in addition that  $H$ is finite, or,
more generally, that
$H$ is a finite union of
intervals of regular cardinals and $H$ includes the cofinalities of its extremes.
 Then the preceding conditions are also equivalent to the following one.
  \begin{enumerate}
    \item[(8)]   
There is an ultrafilter $D$ such that  $K_D ^{Reg}= H$.  
  \end{enumerate}
 \end{corollary}

 \begin{proof} 
(1) $\Rightarrow $  (2) is trivial.

(2) $\Rightarrow $  (3) Let $\mathcal T$ be the class of the regular cardinals
which are not in $H$. Notice that  every member of $\mathcal T$ is
$[ \nu, \nu]$-compact,
for every $\nu \in H$. 
By (2), every product of members of $\mathcal T$, too, is
$[ \nu, \nu]$-compact,
for every $\nu \in H$. 
For every $\nu \in H$, by 
\cite{sak78} or \cite[Theorem 3.4]{Ca}, there exists a
$\nu$-decomposable (equivalently, $(\nu, \nu)$-regular, since $\nu$ is regular)
 ultrafilter   
$D_\nu$ such that every member of $\mathcal T$ is
$D_\nu$-compact. 
It is easy to see that if $\kappa$ is a regular cardinal, $D$ is an ultrafilter,
 and the space $\kappa$ is $D$-compact, then $D$ is $\kappa$-descendingly
complete 
(see the proof of \cite[Proposition 1]{tproc}, though the result is stated with different terminology),
equivalently, $D$ is not $\kappa$-decomposable.
In the case at hand, this shows that if $\kappa $ is regular and does not belong to 
$  H$, then $\kappa \not \in K _{D_\nu}$.
This means exactly that $ K _{D_\nu} ^{Reg} \subseteq H$. Since 
$D_ \nu$ is $\nu$-decomposable, then $ \nu \in K _{D_\nu}$,
hence $D_\nu$ witnesses (3).  

(3$'$) is a restatemetn of (3).

(3) $\Rightarrow $  (4) We have to show that every product 
$Y$ as in (4)
is $[ \nu , \nu]$-compact, for every $\nu \in H$.
So let $ \nu \in H$, and let $D_ \nu$ be given by (3).
Since $K _{D_\nu} ^{Reg}\subseteq H$,
then no factor of $Y$ has a (pseudo-)gap in $K _{D_\nu} ^{Reg}$.
By Theorem \ref{thm} (4) $\Rightarrow $  (1) applied to $D_ \nu$,
each factor of $Y$ is $D_ \nu$-compact, hence also $Y$ is    
$D_ \nu$-compact.
Hence $Y$ is $[ \nu, \nu]$-compact, since 
$\nu \in K _{D_\nu} ^{Reg}$ and, as we mentioned,
\ref{thm} (1) $\Rightarrow $  (2) holds for every topological space. 
(Of course, we are repeating here some arguments from the proof of 
Corollary \ref{corthm}, and the implication can be also obtained as a consequence of  
\ref{corthm}).

(4) $\Rightarrow $  (1) is trivial, since if a GO space
is $[ \nu, \nu]$-compact, then it has no (pseudo-)gap of type $\nu$.
Cf. the proof of \ref{thm} (2) $\Rightarrow $  (3). 

Hence (1)-(4) are all equivalent.

(3) $\Rightarrow $  (5) is similar to  (3) $\Rightarrow $  (4).

(5) $\Rightarrow $  (5$'$) is trivial.

(5$'$) $\Rightarrow $  (1), again, is trivial, since if a GO space
is $[ \nu, \nu]$-compact, then it satisfies all 
the conditions (2)-(4), (6)-(7) in \ref{thm}, for that given $\nu$. 

Hence (1)-(5) are all equivalent.

The equivalence of (3) and (6) is purely set-theoretical. (6) $\Rightarrow $  (3)
is trivial since $H$ consists only of regular cardinals. For the converse, suppose that  $F= \{ \nu_1, \dots, \nu_n  \} $,
and let $D _{\nu_1} $, \dots,  $D _{\nu_n} $ be the corresponding
ultrafilters given by (3). We are going to show that
 $D_F = D _{\nu_1} \times  \dots   \times D _{\nu_n}$ works
in (6). Recall that, for $D$, $D'$ ultrafilters, say, over $I$, $I'$,
their \emph{product}  $D \times D' $
 is the ultrafilter over $I \times I'$  defined
by $ Z \in D \times D'$ if and only if 
$\{ i \in I \mid \{i' \in I'  \mid (i, i') \in Z \} \in D' \} \in D$.
It follows from the last statement in \cite[Proposition 7.1]{mru}
that, for every regular cardinal $\nu$, the following holds: 
$\nu \in K _{D \times D'} $ 
if and only if either
$\nu \in K _{D } $, or 
$\nu \in K _{D'} $ 
 (the result in \cite[Proposition 7.1]{mru} is stated in terms
of $(\nu,\nu)$-regularity, but, when $\nu$ is regular, this is equivalent to
$\nu$-decomposability). 
In other words, 

(*) $K _{D \times D'} ^{Reg}   = K _{D } ^{Reg} \cup K _{ D'} ^{Reg}$,
for every pair of ultrafilters $D$ and $D'$. 

In particular, 
if ${D_F}$ is defined as above,
then 
$K _{D_F} ^{Reg}   =
\bigcup _{\nu= \nu_1, \dots, \nu_n} K _{D_\nu} ^{Reg} $. 
By (3) we have that  $K _{D_\nu} ^{Reg}\subseteq H$,
for $\nu= \nu_1, \dots, \nu_n$, hence
$K _{D_F} ^{Reg} \subseteq H$.
Moreover, $ \nu \in K _{D_\nu}$,
for $\nu= \nu_1, \dots, \nu_n$, hence
$F \subseteq K _{D_F}^{Reg}$. We have proved (3) $\Rightarrow $  (6).

The equivalence of (3) and (7), too, can be proved in a set-theoretical way,
by a small variation on \cite[Proposition 7.6 (a) $\Leftrightarrow $  (b)]{mru},
 considering the case $\chi = \omega $  there,
since every ultrafilter is $ \omega$-complete. 
Notice that the proof of
\cite[Proposition 7.6]{mru}
uses $( \nu , \nu)$-regularity even for $\nu$ singular,
hence  a proof along those lines 
needs
the assumption that
$H$ includes the cofinalities of the extremes of $F$.
In fact, the assumption is necessary,
as we shall show in Remark \ref{bbrmk}.

We shall give here a 
proof   of (3) $\Rightarrow $  (7)
  with a stronger topological flavor, and which relies
heavily on \cite{Ca}.

As above, (7) $\Rightarrow $  (3) is trivial. To get the converse, 
we shall prove (1) $\Rightarrow $  (7). First, we need a lemma,
an easy extension of \cite[Corollary 1.8(iii)]{Ca}.
\renewcommand{\qedsymbol}{$\Box_{to\ be\ continued}$}
\end{proof}

\begin{lemma} \labbel{caiclem}
Suppose that  $\mathcal T$ is a class of topological spaces,
and $\lambda$, $\mu $ are infinite cardinals. If every product
of members of $\mathcal T$ is 
both  $[ \cf \mu, \cf\mu ]$-compact,
and
$[\nu, \nu]$-compact, for every regular
$\nu$ with $\lambda \leq \nu \leq \mu$,   
then every product of members of $\mathcal T$ is
$[ \lambda, \mu ]$-compact.
 \end{lemma}

\begin{proof} 
The particular case 
$\cf \mu \geq \lambda $ is \cite[Corollary 1.8(iii)]{Ca},
and the lemma can be proved in a similar way.
Otherwise, in case $\cf \mu < \lambda $,
we can apply \cite[Corollary 1.8(iii)]{Ca}
for every regular $\mu' < \mu $,
getting that every product of members of $\mathcal T$ 
is  $[ \lambda, \mu' ]$-compact, for every regular $\mu' < \mu $,
and this, together 
with   $[ \cf \mu, \cf\mu ]$-compactness immediately implies
$[ \lambda, \mu ]$-compactness.
\end{proof}

\begin{proof}[Proof of \ref{gocorbb} (continued)]
We have promised a proof of  (1) $\Rightarrow $  (7).
So, let 
$F = Reg \cap  \bigcup _{p \in P} [ \lambda _p, \mu _p] $,
with $P$ finite, and $\cf \mu_p \in H$, for every $p \in P$.  
 We have from (1) that every product of  
$H$-compact spaces
is 
still 
$H$-compact.
Applying Lemma \ref{caiclem} to the family $\mathcal T$ of all
$H$-compact spaces,
since $F $ is assumed to be 
a subset of $  H$
and since $H$ includes the cofinalities of the extremes of $F$, we get that, for every $p \in P$,
every product of $H$-compact spaces is 
$[ \lambda_p, \mu_p ]$-compact.
By \cite[Theorem 3.4]{Ca},
for every $p \in P$,
there is 
a $( \lambda_p, \mu_p )$-regular ultrafilter $D_p$
such that every $H$-compact space is $D_p$-compact.  
As in the proof of (2) $\Rightarrow $  (3), 
 for every regular cardinal $\kappa \not\in H$,
the space $\kappa$ with the order topology is $H$-compact,
hence $D_p$-compact,  and then,  using again
\cite[Proposition 1]{tproc}, we get that 
$D_p$ is not $\kappa$-decomposable, thus 
$\kappa \not\in K _{D_p} ^{Reg} $.
Since this holds for every regular $\kappa \not\in H$, then
  $ K _{D_p}^{Reg} \subseteq H $.
Enumerating $P$ as $\{ p_1, \dots, p_n \}$, and  putting  
$D_F = D _{p_1} \times  \dots   \times D _{p_n}$,
we get $ K _{D_F}^{Reg} \subseteq H $,
by using (*) above.
Moreover, a $( \lambda, \mu )$-regular ultrafilter 
is
$\nu$-decomposable, for every regular $\nu$ such that 
$\lambda \leq \nu \leq \mu$ (see, e.~g., \cite[Property 1.1(xii)]{mru}),
thus $[ \lambda_p, \mu_p ] \cap Reg \subseteq K _{D_p}$,
for every $p \in P$.
Applying again (*),
we get that 
$F \subseteq K _{D_F}$.
In conclusion, $D_F$
witnesses (7).   

To complete the proof of  Corollary \ref{gocorbb},
it remains to prove 
(7) $\Leftrightarrow $    (8)
under the additional assumption, 
but this is trivial,  by taking 
$F=H$
(of course, in case $H$ is finite, it is easier to use (6)). 
\end{proof}  
 
\begin{remark} \labbel{bbrmk}
In general, the additional assumption before condition (8)
in Corollary \ref{gocorbb} 
is necessary in order to prove the equivalence of (8)
with the other conditions.

First, we are going to show
that (3) does not necessarily imply (8), when $H$ is infinite.
Suppose that there are infinitely many measurable cardinals
$\mu _1< \mu_2< \dots$, and let $H= \{ \mu_n \mid n \in \omega \}$.
If $\mu $ is a measurable cardinal, there is a $\mu $-complete  uniform 
ultrafilter $D_ \mu $ over $\mu $, thus $K _{D_\mu}= \mu=K _{D_\mu}^{Reg}$.
This implies that condition (3) in  \ref{gocorbb}  holds for the above
 $H$.

However, if we further assume GCH,
there is no ultrafilter $D$ such that 
$K _{D}^{Reg}=H$. Indeed, the last paragraph
in \cite[Section 7]{mru}  shows that if $D$  is such an ultrafilter,
then $D$ is $( \kappa , \kappa ) $-regular, 
 for $\kappa= \sup _{n \in \omega }  \mu_n $. 
But then \cite[Corollary 5.8]{mru} implies that 
$D$ is either $\cf \kappa $-decomposable,
or $\kappa ^+$-decomposable, thus 
$K _{D}^{Reg} $ strictly contains $ H$, a contradiction
(notice that $\cf \kappa = \omega $).

There is also a similar counterexample in which $H$ is an interval,
but $H$ does not include the cofinality of its upper extreme.
Suppose that $\kappa$ is $\kappa ^{+n} $-compact, 
for every $ n \in \omega$.
Here $\kappa ^{+n} $ is $\kappa ^{+\dots+} $ with $n$ occurrences of ``$+$''.  
Let $\kappa ^{+ \omega }= \sup _{n \in \omega }  \kappa ^{+n} $
and $H=[ \kappa , \kappa ^{+ \omega })= Reg \cap [ \kappa , \kappa ^{+ \omega }]$.
By $\kappa ^{+n} $-compactness, there is some $\kappa$-complete
$( \kappa, \kappa ^{+n})$-regular 
 ultrafilter $D_n$, thus 
$K _{D_n} = K _{D_n}^{Reg}=[ \kappa, \kappa ^{+n}] $.
Hence $H$ satisfies condition (3) in  \ref{gocorbb}.
However, there is no ultrafilter $D$ such that $K _{D}^{Reg}=H$.
Indeed, by \cite[Theorem 5.7 and Corollary 5.8]{mru}, such a $D$
would be either  $ \kappa ^{+ \omega +1}$-decomposable, or 
$\cf \kappa ^{+ \omega }$-decomposable, that is, $ \omega$-decomposable;
but both the above assertions contradict  $K _{D}^{Reg}=H$.
Thus (8) fails. 

Notice  that this last counterexample also shows
that the assumption that $H$ includes the cofinalities of the extremes of $F$
is necessary in (7): just take $F=H$ to make (7) fail. 
 \end{remark}

Corollary  \ref{gocorbb} 
shows that certain properties of products of GO spaces
depend heavily on the set theoretical universe in which we work.
  If there is no inner model with a measurable cardinal, then,
for every ultrafilter $D$, 
$K_D$ is always an interval of cardinals
with minimum $ \omega$, by results from Donder
\cite{Do}, which extends and generalizes former results by M. Benda, C. C. Chang,
 R. B. Jensen, J. Ketonen, K. Prikry, J. H. Silver, among many others. 
See \cite[p. 363]{mru} for details.
On the other hand, modulo some large cardinal consistency 
assumptions, it is possible to have an ultrafilter $D$ for which
$K_D ^{Reg} = \{ \omega, \omega _{ \omega +1} \} $, 
Ben-David and Magidor \cite{BM}, Apter and  Henle \cite{AH}.
This is just an example concerning relatively small cardinals; many
results are known, and nevertheless
deep problems are still open about the possible values
that the set $K_D$ can assume.
 See Problem 6.8 in \cite{mru}, and the comments below it.
By Corollary \ref{gocorbb} (e.~g., (1) $\Leftrightarrow $ (3)), all these problems affect the behavior of GO spaces with respect to products,
thus the use of ultrafilters in the results of the present section proves to be irreplaceable.
We shall explicitly state just the example dealing with the smallest possible
cardinals.

\begin{corollary} \labbel{gocorbbb}
(Assuming the consistency of a strongly compact cardinal)
If $H = \{ \omega, \omega _{ \omega +1}  \}$,
then the preservation of $H$-compactness 
under products in the class of GO spaces
is both relatively consistent and independent from the axioms of ZFC.
 \end{corollary}

 \begin{proof} 
Immediate from Corollary \ref{gocorbb}(1) $\Leftrightarrow $  (8), and
the just
 mentioned results by Donder,  Ben-David and Magidor, and Apter and  Henle.
\end{proof}

\begin{remark} \labbel{bbbrmk}    
Notice that 
$H = \{ \omega, \omega _{ \omega +1}  \}$ is the ``lowest'' set
for which independence  can occur.  Indeed, \cite[Corollary 1.8(ii)]{Ca} 
proves that if every product of members of a family $\mathcal T$ is 
$[ \lambda ^+, \lambda ^+]$-compact, then  
every product of members of $\mathcal T$ is 
$[ \lambda , \lambda ]$-compact.
Iterating, we get that if all products of  members of $\mathcal T$
(whether $\mathcal T$ consists of GO spaces or not)  
 are
$[ \omega _n, \omega_n]$-compact,
then all such products are 
$[ \omega _i, \omega_i]$-compact,
for all $i \leq n$, hence 
also initially $ \omega_n$-compact.  
Moreover, if there is some infinite cardinal $\lambda$ 
such that every product of members of some family $\mathcal T$ is 
$[ \lambda , \lambda ]$-compact, then the smallest such $\lambda$
is either
$ \omega$, or a measurable cardinal (hence it cannot be 
$ \omega _{ \omega +1} $). 
This is a slight generalization (with essentially the same proof)
of \cite[Corollary 1.8(i)]{Ca}.
\end{remark}

\section{Relationships between $H$-compactness and omission of gaps in $H$}  \labbel{equivcpn}

The results in the present paper suggest
that it is interesting to study the following
 property,
depending on a class $H$ of infinite regular cardinals.

\smallskip

(**)$_H$ For every  GO space $X$,
$X$ is $[H]$-compact
if and only if $X$  has no (pseudo-)gap of type in $H$. 

\smallskip

In general,
(**)$_H$ is false; for example, take 
$H=\{ \lambda  \}$ with
$ \omega <\lambda \leq \mathfrak c$;
then $\mathbb R \times \mathbb Z $ 
with the lexicographic order
furnishes a counterexample to (**)$_H$.

However, (**)$_H$ holds in many interesting cases.
The simplest affirmative 
case is given already by Gulden, Fleischman and Weston result
in the case of LOTS, or Corollary \ref{initial} (2) $\Leftrightarrow $  (4)
here for general GO spaces.
In this sense, the corollary   states that 
(**)$_H$ holds in case
$H= Reg \cap [ \omega, \lambda ]$. 
Another example is given by
Corollary \ref{gocorbb}, which implies
that (**)$_H$ holds in case 
$H=  K _{D}^{Reg} $, for some ultrafilter $D$.

Notice that, for $\lambda$ regular, 
if a GO space is $[\lambda, \lambda ]$-compact,
then, trivially, it has no (pseudo-)gap of type $\lambda$;
hence, in order to prove an instance of
 (**)$_H$, it is enough to prove 
the non trivial implication. In order to give some results about
 (**)$_H$ we need some further definitions.

Suppose that $\lambda$ is an infinite 
 cardinal, $X$ is a linearly ordered set, and $Y \subseteq X$.
If $Y$ has no maximum, we say 
that $Y$ is \emph{$\lambda$-full (in $X$) on the right}
if, for every $y \in Y$, 
$ \bigcup _{y' \in Y} (y, y')_X $ 
has cardinality $\geq \lambda$. 
The subscript $_X$ in 
$(y, y')_X$ is intended to mean that the interval is evaluated in $X$, 
that is, 
$(y, y')_X= \{x \in X \mid y < x < y'\} $.
When $X$ is clear from the context, we shall omit it.
Symmetrically,
if $Y$ has no minimum, 
 $Y$ is \emph{$\lambda$-full (in $X$) on the left}
if, for every $y \in Y$, 
$ \bigcup _{y' \in Y} (y', y)_X $ 
has cardinality $\geq \lambda$. 
For sake of brevity, if $\alpha$ is an ordinal, we say that a linear order $Y$ is 
\emph{$ \alpha \pm$ordered} in case $Y$ is order-isomorphic either to $ \alpha $,
or to $\alpha \breve \ $, i.e., $ \alpha $ with the reversed order.
 
The next lemma shows that $\lambda$-full subsets always exist,
for linear orders of cardinality $\geq\lambda$.

\begin{lemma} \labbel{omega}
Suppose that $\lambda$ is an infinite  regular cardinal.
Then every linearly ordered set $X$ of cardinality $\lambda$ 
either has a $\lambda\pm$ordered subset, or  a $\lambda$-full 
 subset of order type $ \omega$.
 \end{lemma}  

\begin{proof}
Suppose that there is $Y \subseteq X$ such that $|Y| = \lambda $
and, for every  $y \in Y$, either 
$(-\infty, y)_Y$ or $(y, \infty)_Y$ 
has cardinality $<\lambda$.   Then an easy induction of length $\lambda$, using the regularity
of $\lambda$, shows that $Y$ (and hence $X$) has a $\lambda\pm$ordered subset.
Otherwise,
for every  $Y \subseteq X$ of cardinality $\lambda $,
there is $y \in Y$ such that both
$(-\infty, y)_Y$ and $(y, \infty)_Y$ 
have cardinality $\lambda$.
Then an induction of length $ \omega$ 
produces a strictly increasing  sequence $(y_n) _{n \in \omega } $
such that $\{ y_n \mid n \in \omega \}$ is
 $\lambda$-full in $X$.
\end{proof}
 
\begin{proposition} \labbel{pair}
Suppose that 
$\nu \leq \lambda$
are regular cardinals.
 Then the following conditions are equivalent.
 \begin{enumerate}
   \item 
Every linearly ordered set of cardinality $ \lambda$ 
has 
either a $ \lambda \pm$ordered subset, or 
a $\nu\pm$ordered $\lambda$-full subset. 
\item
For every GO space $X$, 
if $X$ has no (pseudo-)gap of type $\nu$ or $\lambda$,
then $X$ is $[ \lambda, \lambda ]$-compact. 
\item
For every LOTS $X$, 
if $X$ has no gap of type $\nu$ or $\lambda$,
then $X$ is $[ \lambda, \lambda ]$-compact.  
  \end{enumerate}

In Conditions (2) and (3) we can equivalently
restrict ourselves to those $X$'s such that $|X| \leq \lambda ^\nu$. 
 \end{proposition}  

\begin{proof} (Sketch)
Recall that
if $\lambda$ is regular, then 
$[ \lambda, \lambda ]$-compactness is equivalent
to
$\CAP_ \lambda $. We shall use this equivalent condition throughout the proof.

(1) $\Rightarrow $  (2) 
Let
$X$ be a GO space,
$Z\subseteq X$, and $|Z|= \lambda $.  
Apply (1) to $Z$, say, $Z$  has  subset 
$Y$ of order type $\nu$ and $\lambda$-full in $Z$. 
Since $X$ has no gap of type $\nu$,
there is in $X$ a supremum $y$ of $Y$, and $y$ belongs to the closure
of $Y$. Then every neighborhood of $y$ contains $\lambda$-many elements from
$Z$, since $Y$ is $\lambda$-full in $Z$; in particular,  $\lambda$-many elements from
$X$, since $X \supseteq Z$.

The other cases are treated in a similar way.

(2) $\Rightarrow $  (3) is trivial.

(3) $\Rightarrow $  (1) We shall sketch a proof of the contrapositive. 
So, suppose that (1) fails, thus there exists a linear order $L$ 
of cardinality $\lambda$ without $ \lambda \pm$ordered subsets and 
without $\nu\pm$ordered $\lambda$-full subsets. 
By Lemma \ref{omega}, necessarily $\nu> \omega $. 
Extend $L$ to another linearly ordered set $X$ in the following way.
Fill each gap having both left and right type $\nu$ by  adding an element
in the middle of the gap. If the gap has left type $\nu$ but a different right type,
``fill'' the gap by adding a copy of $ \omega$. In the symmetric situation,
add a copy of $ \omega \breve \ $. This procedure destroys all gaps of
type $\nu$.   
We also want to ``forbid'' the existence (in $X$) of complete accumulation 
points of $L$. A candidate for such an accumulation point is some 
$ \ell \in L$ such that, say, 
$(\ell, \infty)_L$  is $\lambda$-full in $L$ on the left.
Then replace $\ell$ with a copy of 
$ \omega$ or, according to the situation,
with a copy of 
$ \mathbb Z $. In the symmetric situation,
it might be necessary to replace $\ell$ with a copy of $ \omega \breve \ $,
instead.

Formally, $X$ can be considered as a subset
of  $L^+ \times \mathbb Z$ with the lexicographical order,
where $L^+$ is the completion of $L$.
It can be checked that the above constructions introduce no new gap of type
$\nu$ or $\lambda$ ($L$ has no gap of type $\lambda$ from the beginning),
and that no point of $X$ is a complete accumulation point of $L$ (here we need the assumption
that L has no $\nu\pm$ordered $\lambda$-full subset
otherwise the above ``fillings'' of gaps could introduce
some complete accumulation point). Since $|L| = \lambda $, then (3) fails. 

In order to prove the last statement it is enough to check that
the $X$ constructed in the last part of the proof has cardinality
$ \leq \lambda ^ \nu$. 
\end{proof}

\begin{corollary} \labbel{proph}
If $H$ is a set of regular cardinals,
then the following holds.
  \begin{enumerate}    
\item 
If $H= \{ \lambda   \} $ 
is a singleton, then (**)$ _{H} $ holds if and only if 
$\lambda$ is either $ \omega$, or a weakly compact cardinal.
\item
If  (**)$ _{H} $ holds, then $\inf H$ 
is either $ \omega$ or a weakly compact cardinal.
\item
If
$\inf H= \omega $, then  (**)$ _{H} $ holds.
  \end{enumerate} 
 \end{corollary} 

\begin{proof}
(1)
 is the particular case $\nu= \lambda $
of Proposition \ref{pair} (2) $\Leftrightarrow $  (1), since 
if $\lambda$ is regular, then every   $ \lambda \pm$ordered subset
is necessarily $\lambda$-full. Recall
that, by a well-known characterization, an uncountable cardinal $\lambda$ is
 weakly  compact if and only if every linear order of cardinality
$\lambda$ has a $\lambda\pm$ordered subset. 
Notice that $ \omega$, too,
satisfies the above property.

(2)
Let $\lambda= \inf H$.
If $\lambda = \omega $, we are done.
Otherwise, suppose by contradiction that $\lambda$
is not weakly compact, thus there is
 linear order $L$ without a $\lambda\pm$ordered subset.
Consider the LOTS $X=L \times \mathbb Z$, with the lexicographic order.
Then $X$ has no gap of type $\lambda$, since $\lambda> \omega $ ; moreover, it
has no gap of type $\nu \in H \setminus \{ \lambda  \} $,
since $|X| = \lambda = \inf H< \nu$.
By (**)$ _{H} $, $X$ is $[ \lambda, \lambda ]$-compact,
a contradiction, since $X$ has 
cardinality $\lambda$ and the discrete topology.

(3) 
Suppose that a GO space $X$
has no (pseudo-)gap of type in $H$.
By Lemma \ref{omega}, for every  
$\lambda \in H$, 
condition (1) in Proposition  \ref{pair} is satisfied, when $\nu= \omega $.   
By Proposition \ref{pair}  (1) $\Rightarrow $  (2),
 $X$ is $[ \lambda, \lambda ]$-compact.
This happens
for every $\lambda \in H$, hence
$X$ is $[H]$-compact.
\end{proof}

\section*{Disclaimer} \labbel{disc}

Though the author has done his best efforts to compile the following
list of references in the most accurate way,
 he acknowledges that the list might 
turn out to be incomplete
or partially inaccurate, possibly for reasons not depending on him.
It is not intended that each work in the list
has given equally significant contributions to the discipline.
Henceforth the author disagrees with the use of the list
(even in aggregate forms in combination with similar lists)
in order to determine rankings or other indicators of, e.~g., journals, individuals or
institutions. In particular, the author 
 considers that it is highly  inappropriate, 
and strongly discourages, the use 
(even in partial, preliminary or auxiliary forms)
of indicators extracted from the list in decisions about individuals (especially, job opportunities, career progressions etc.), attributions of funds, and selections or evaluations of research projects.

\bibliographystyle {amsalpha}

\begin{thebibliography}{GFW70}

\bibitem[AH92]{AH}
Arthur~W. Apter and James~M. Henle, \emph{On box, weak box and strong
  compactness}, The Bulletin of the London Mathematical Society \textbf{24}
  (1992), no.~6, 513--518.

\bibitem[AU29]{AU}
P.~Alexandroff and P.~Urysohn, \emph{{M\'emoire sur les espaces topologiques
  compacts d\'edi\'e \`a Monsieur D.~Egoroff.}}, {Verhandelingen Amsterdam 14,
  No.1, 93 S.}, 1929.

\bibitem[BDM86]{BM}
Shai Ben-David and Menachem Magidor, \emph{The weak {$\square^\ast$} is really
  weaker than the full {$\square$}}, The Journal of Symbolic Logic \textbf{51}
  (1986), no.~4, 1029--1033.

\bibitem[Ber70]{ber70}
Allen~R. Bernstein, \emph{A new kind of compactness for topological spaces},
  Polska Akademia Nauk. Fundamenta Mathematicae \textbf{66} (1969/1970),
  185--193.

\bibitem[BL02]{BL}
Harold~R. Bennett and David~J. Lutzer, \emph{Recent developments in the
  topology of ordered spaces}, Recent progress in general topology, {II}
  (Miroslav Hu{\v{s}}ek and Jan van Mill, eds.), North-Holland, Amsterdam,
  2002, pp.~83--114.

\bibitem[Cai99]{Ca}
Xavier Caicedo, \emph{The abstract compactness theorem revisited}, Logic and
  foundations of mathematics ({F}lorence, 1995) (Andrea Cantini, Ettore Casari,
  and Pierluigi Minari, eds.), Synthese Lib., vol. 280, Kluwer Acad. Publ.,
  Dordrecht, 1999, Selected papers from the 10th International Congress of
  Logic, Methodology and Philosophy of Science held in Florence, August 1995,
  pp.~131--141.

\bibitem[CK90]{CK}
C.~C. Chang and H.~J. Keisler, \emph{Model theory}, third ed., Studies in Logic
  and the Foundations of Mathematics, vol.~73, North-Holland Publishing Co.,
  Amsterdam, 1990.

\bibitem[Don88]{Do}
Hans-Dieter Donder, \emph{Regularity of ultrafilters and the core model},
  Israel Journal of Mathematics \textbf{63} (1988), no.~3, 289--322.

\bibitem[GF99]{gf99}
Salvador Garc{\'{\i}}a-Ferreira, \emph{On two generalizations of
  pseudocompactness}, Topology Proceedings \textbf{24} (1999), 149--172 (2001).

\bibitem[GFS97]{GFS}
Salvador Garc{\'{\i}}a-Ferreira and Manuel Sanchis, \emph{On {$C$}-compact
  subsets}, Houston Journal of Mathematics \textbf{23} (1997), no.~1, 65--86.

\bibitem[GFW70]{GFW}
Samuel~L. Gulden, William Fleischman, and J.~H. Weston, \emph{Linearly ordered
  topological spaces}, Proceedings of the American Mathematical Society
  \textbf{24} (1970), 197--203.

\bibitem[GS75]{GS}
John Ginsburg and Victor Saks, \emph{Some applications of ultrafilters in
  topology}, Pacific Journal of Mathematics \textbf{57} (1975), no.~2,
  403--418.

\bibitem[Kan86]{kan86}
Akihiro Kanamori, \emph{Finest partitions for ultrafilters}, The Journal of
  Symbolic Logic \textbf{51} (1986), no.~2, 327--332.

\bibitem[Lip96]{tproc}
Paolo Lipparini, \emph{Productive {$[\lambda,\mu]$}-compactness and regular
  ultrafilters}, Topology Proceedings \textbf{21} (1996), 161--171.

\bibitem[Lip10]{mru}
\bysame, \emph{More on regular and decomposable ultrafilters in {ZFC}}, MLQ.
  Mathematical Logic Quarterly \textbf{56} (2010), no.~4, 340--374.

\bibitem[Lip11a]{mpcap}
\bysame, \emph{Every weakly initially {$\mathfrak m$}-compact topological space
  is {$\mathfrak m$}{PCAP}}, Czechoslovak Mathematical Journal \textbf{61(136)}
  (2011), no.~3, 781--784.

\bibitem[Lip11b]{tapp2}
\bysame, \emph{More generalizations of pseudocompactness}, Topology and its
  Applications \textbf{158} (2011), no.~13, 1655--1666.

\bibitem[Lip12]{tproc2}
\bysame, \emph{Some compactness properties related to pseudocompactness and
  ultrafilter convergence}, Topology Proceedings \textbf{40} (2012), 29--51.

\bibitem[Lip13]{goic}
\bysame, \emph{Initial $\lambda$-compactness in linearly ordered spaces}, 2013,
  arXiv:1306.1715.

\bibitem[Lip14]{sssr}
\bysame, \emph{Topological spaces compact with respect to a set of filters},
  Central European Journal of Mathematics \textbf{12} (2014), no.~7, 991--999.

\bibitem[Nag85]{nag85}
Jun-iti Nagata, \emph{Modern general topology}, second ed., North-Holland
  Mathematical Library, vol.~33, North-Holland Publishing Co., Amsterdam, 1985.

\bibitem[NV87]{NV}
P.~J. Nyikos and J.~E. Vaughan, \emph{Sequentially compact,
  {F}ranklin-{R}ajagopalan spaces}, Proceedings of the American Mathematical
  Society \textbf{101} (1987), no.~1, 149--155.

\bibitem[Pur73]{purisch}
Steven~D. Purisch, \emph{On the orderability of {S}tone-\v {C}ech
  compactifications}, Proceedings of the American Mathematical Society
  \textbf{41} (1973), 55--56.

\bibitem[Ret93]{ret93}
Teklehaimanot Retta, \emph{Some cardinal generalizations of pseudocompactness},
  Czechoslovak Mathematical Journal \textbf{43(118)} (1993), no.~3, 385--390.

\bibitem[Sak78]{sak78}
Victor Saks, \emph{Ultrafilter invariants in topological spaces}, Transactions
  of the American Mathematical Society \textbf{241} (1978), 79--97.

\bibitem[Ste84]{ste84}
R.~M. Stephenson, Jr., \emph{Initially {$\kappa$}-compact and related spaces},
  Handbook of set-theoretic topology (Kenneth Kunen and Jerry~E. Vaughan,
  eds.), North-Holland, Amsterdam, 1984, pp.~603--632.

\bibitem[SV74]{sv74}
R.~M. Stephenson, Jr. and J.~E. Vaughan, \emph{Products of initially
  {$m$}-compact spaces}, Transactions of the American Mathematical Society
  \textbf{196} (1974), 177--189.

\bibitem[Vau84]{vau84}
Jerry~E. Vaughan, \emph{Countably compact and sequentially compact spaces},
  Handbook of set-theoretic topology (Kenneth Kunen and Jerry~E. Vaughan,
  eds.), North-Holland, Amsterdam, 1984, pp.~569--602.

\end{thebibliography}

\def\cprime{$'$} \def\cprime{$'$} \def\cprime{$'$}
\providecommand{\bysame}{\leavevmode\hbox to3em{\hrulefill}\thinspace}
\providecommand{\MR}{\relax\ifhmode\unskip\space\fi MR }
\providecommand{\MRhref}[2]{%
  \href{http://www.ams.org/mathscinet-getitem?mr=#1}{#2}
}
\providecommand{\href}[2]{#2}

\end{document}